\documentclass{rmmcart}

\usepackage{amsmath,amssymb,amsxtra,color,calligra,mathrsfs,comment,url}
\usepackage[all]{xy}
\usepackage{soul} 
\usepackage{tikz,tikz-cd,enumerate}
\usetikzlibrary{matrix,arrows,decorations.pathmorphing}

\newcommand{\simm}[1]{\underset{#1}{\sim}}

\numberwithin{equation}{subsection}
\newtheorem{theorem}{Theorem}[section]
\newtheorem{lemma}[theorem]{Lemma}
\newtheorem{corollary}[theorem]{Corollary}
\newtheorem{proposition}[theorem]{Proposition}

\theoremstyle{definition}

\title{Unimodular rows over monoid extensions of overrings of polynomial rings}

\author{Maria A. Mathew} 
\address{Department of Mathematics, Indian Institute of Technology Bombay, Powai, Mumbai, 400076, India}
\email{maria@math.iitb.ac.in}

\author{Manoj K. Keshari} 
\address{Department of Mathematics, Indian Institute of Technology Bombay, Powai, Mumbai, 400076, India}
\email{keshari@math.iitb.ac.in}

\date{09, 15, 2020}


\begin{document}
	
\begin{abstract}
 Let $R$ be a commutative Noetherian ring of dimension $d$ and $M$ a commutative cancellative torsion-free seminormal monoid. Then (1) Let $A$ be a ring of type $R[d,m,n]$ and $P$ be a projective $A[M]$-module of rank $r \geq max\{2,d+1\}$. Then the action of $E(A[M] \oplus P)$ on $Um(A[M] \oplus P)$ is transitive and (2) Assume $(R, m, K)$ is a regular local ring containing a field $k$ such that either $char$ $k=0$ or $ char$ $k = p$ and $tr$-$deg$ $K/\mathbb{F}_p \geq 1$. Let $A$ be a ring of type $R[d,m,n]^*$ and $f\in R$ be a regular parameter. Then all finitely generated projective modules over $A[M],$ $A[M]_f$ and $A[M] \otimes_R R(T)$ are free. When $M$ is free both results are due to Keshari and Lokhande \cite{MKK&SAL-MR3153610}.
\end{abstract}

\maketitle

\section{Introduction}

In this paper$,$ rings are commutative Noetherian with unity and modules are finitely generated.

Let $R$ be a ring of dimension $d$. A ring $A$ is called of type $R[d,m,n]$ if it is an overring of a polynomial ring $R[X_1, \ldots, X_m]$ given by $A=R[X_1, \ldots, X_m, f_1(l_1)^{-1}, \ldots, f_n(l_n)^{-1}],$ where $f_i \in R[T]$ and  either $l_i$ is
some indeterminate $X_{i_j}$ for all $i$ or $R$ contains a field $k$ and $l_i =\sum_{j=1}^{m}\alpha_{ij}X_j - r_i,$ where $(\alpha_{i1},\ldots,\alpha_{im}) \in k^m-\{0\} $ and $r_i \in R$ for all $i, j$.  We say that $A$ is a ring of type $R[d,m,n]^*$ if we further assume that $f_i(T) \in k[T] $ and $r_i \in k$ for all $i$.

A monoid $M$ is  \textit{cancellative} if $ax=ay$ implies $x=y$ for $a,x, y \in M$.  A cancellative monoid is \textit{torsion-free} if for $x,y\in M$ and $n > 0,$ $x^n = y^n$ implies $x=y$.

A projective $R$-module $P$ is \textit{cancellative} if for any projective $R$-module $Q,$ $P \oplus R^n \simeq Q \oplus R^n$ implies $P \simeq Q$. Equivalently$,$ by Bhatwadekar (\cite{Bhatw-MR1690788}$,$ Proposition 2.17)$,$ $Aut(P\oplus R^n)$ acts transitively on the set $Um(P\oplus R^n)$ of unimodular elements for all $n\geq 1$.  

In \cite{MKK&SAL-MR3153610}$,$ Keshari and Lokhande proved that \textit{if $A$ is a ring of type $R[d,m,n]$ and  $P$ is a projective $A$-module of rank $ \geq max \{2, dim R+1\},$ then $E(A \oplus P)$ acts transitively on $Um(A \oplus P)$}.  This result was proved by Dhorajia and Keshari \cite{AMD&MKK-MR2564858}$,$ when $l_i$ is
some indeterminate $X_{i_j}$ for all $i$.  We extend above result to monoid algebra $A[M]$ as follows.

\begin{theorem}
	Let $A$ be a ring of type $R[d,m,n]$ and $M$ be a commutative cancellative torsion-free seminormal monoid. Let $P$ be a projective $A[M]$-module of rank $r \geq \max \{2, d+1\}$. Then $E(A[M] \oplus P)$ acts transitively on $Um(A[M] \oplus P)$. In particular$,$ P is cancellative.
\end{theorem}

Gabber \cite{Gabber-MR1891173} proved that \textit{for a field k$,$ all finitely generated projective $k[0,m,n]^*$- modules are free}. It was generalized by Keshari and Lokhande \cite{MKK&SAL-MR3153610} for projective $R[d,m,n]^*$-modules, where $(R, m, K)$ is a regular local ring containing a field $k$ such that either \textit{char} $k = 0$ or \textit{char} $k = p$ and $tr$-$deg$ $K/\mathbb{F}_p \geq 1$. 

Gubeladze \cite{Gubeladze2-MR937805} proved that \textit{for a principal ideal domain R and a commutative cancellative torsion-free seminormal monoid $M,$ all finitely generated projective $R[M]$-modules are free.} In developing an algebraic approach to Gubeladze's proof$,$ Swan \cite{Swan-MR1144038} extended Gubeladze result to rings of higher dimension via introduction of a class of domains $\mathcal{R}_n$ with $n>0$ such that if $R \in \mathcal{R}_n,$ then
\begin{enumerate}
	\item the localization of $R$ with respect to any maximal ideal of $R$ belongs to $\mathcal{R}_n,$
	\item $R(X) \in \mathcal{R}_n,$ where $R(X)$ is the localization of $R[X]$ with respect to the multiplicative set consisting of  all monic polynomials in $X,$
	\item if $R$ is local$,$ then all projective modules of rank $n$ over $R[X, X^{-1}]$ are free.
\end{enumerate}
Let $\mathcal{P}_n(R)$ denote the isomorphism class of finitely generated projective $R$-modules of constant rank $n$. Swan \cite{Swan-MR1144038} proved the result: \textit{Let $R \in \mathcal{R}_n$ and $M$ be a commutative cancellative torsion-free seminormal monoid with $U$ as the group of units of $M$. If $\mathcal{P}_n(R) \rightarrow \mathcal{P}_n(R[U])$ is onto$,$ then $\mathcal{P}_n(R) \rightarrow \mathcal{P}_n(R[M])$ is onto$,$ i.e.$,$ all projective $R[M]$-modules of rank $n$ are extended from $R$.} 

Let us recall Quillen's conjecture \cite{Quillen-MR427303}$,$ $\mathcal{Q}_n$\textit{: If $(A, \mathfrak{m})$ is a regular local ring of dimension $n$ and $u \in \mathfrak{m} \smallsetminus \mathfrak{m}^2,$ then all projctive $A_u$-modules are free}.  Bhatwadekar and Rao \cite{Bhatrao-MR709584} proved that $\mathcal{Q}_n$ is true when $R$ is a regular $k$-spot$,$ i.e.$,$ when $R$ is the localization of some affine $k$-algebra at a regular prime ideal. More generally$,$ they proved the result: \textit{Let $(R,\mathfrak m)$ be a regular $k$-spot with infinite residue field and $f\in \mathfrak m \smallsetminus \mathfrak m^2$ be a regular parameter. If $B$ is one of $R,$ $R(T)$ or $R_f,$ then projective modules over $B[X_1,\ldots,X_n,Y_1^{\pm 1},\ldots,Y_m^{\pm 1}]$ are free.}

Swan \cite{Swan-MR1144038} used the above result and proved that \textit{all localizations of regular affine algebras over fields belong to $\mathcal{R}_n$ for all $n>0$.} This in conjugation with Popescu's result (\cite{Popescu-MR986438}$,$ Theorem 3.1) gives an even bigger class of rings $R[d,m,n]^*$ in $\mathcal{R}_n$. Our fourth section deals with such class of domains and yields the following

\begin{theorem}
	Let $A$ be a ring of type $R[d,m,n]^*,$ where $(R, \mathfrak{m},K)$ is a regular local ring containing a field $k$ such that either char $k = 0$ or char $k=p$ and tr-deg $K/\mathbb{F}_p \geq 1$. Let $M$ be a commutative cancellative torsion-free seminormal monoid. Then
	\begin{enumerate} 
		\item all projective $A[M]$ and $A[M] \otimes_R R(T)$-modules are free;
		\item if $f, g\in R$ form a part of a regular system of parameters$,$ then all projective $A[M]_f$ and $A[M]_{fg}$-modules are free;
		\item if $g_1, \ldots, g_t \in R$ form a part of a regular system of parameters$,$ then projective $A[M]_{{g_1}\ldots {g_t}}$-modules of rank $\geq t$ are free. 
	\end{enumerate}  
\end{theorem}

Though one should note that the above result doesn't hold for rings of the type $R[d,m,n],$ counterexamples have been provied in \cite{MKK&SAL-MR3153610}.


\section{Preliminaries}

Given a projective $R$-module $P,$ $Um(P)$ denotes the set of unimodular elements of $P$ and $E(P\oplus R$) denotes the subgroup of $Aut(P\oplus R)$ generated by transvections. See \cite{MKK&SAL-MR3153610} for basic definitions.

Transvections define an action on the set of unimodular elements $Um(P)$. For $p, q \in Um(P),$ the notation $p \simm{P} q$ means that they are in the same orbit via the $E(P)$ action. When $P$ is free$,$ the action translates to that of multiplication on the right.
The equivalence of $p$ and $q$ is loosely written as $p \simm{R} q,$ where $P$ is understood. 

This paper generalizes the following theorem by Gubeladze (\cite{Gubeladze-MR3853049}$,$ Theorem 1.1)

\begin{theorem}\label{t1}
	Let $R$ be a ring of dimension $d$ and $M$ be a commutative cancellative monoid. Then $E_r(R[M])$ acts transitively on $Um_r(R[M]),$ if $r \geq \max\{3, d+2\}$.
\end{theorem}

The following is a slightly modified version of a lemma due to Lindel (\cite{Lindel-MR1322406}$,$ Lemma 1.1) and will follow from (\cite{AMD&MKK-MR2564858}$,$ Lemma 3.9).

\begin{lemma}\label{l1}
	Let $P$ be a projective $R[d,m,n]$-module of rank $r$. Then there exists an $s \in R,$ which satisfies the following
	\begin{enumerate}
		\item $P_s$ is free$,$
		\item there exists $p_1, \ldots p_r \in P$ and $\phi_1, \ldots, \phi_r \in P^*$ such that
		$ (\phi_i(p_j)) = diagonal(s,\ldots, s),$
		\item $sP \subset p_1R + \ldots + p_rR,$
		\item the image of $s$ in $R_{red}$ is a non-zerodivisor$,$
		\item $(0:sR) = (0: s^2R).$   
	\end{enumerate}
\end{lemma}

Following is due to Dhorajia and Keshari (\cite{AMD&MKK1-MR2826426}$,$ Lemma 3.3)

\begin{lemma}\label{l2}
	Let $B$ be a ring and $P$ be a projective $B$-module of rank $r$. Let  $s\in B$ be as in the above lemma and $B' = \frac{B[X]}{(X^2 - s^2X)}$. Assume that $E_{r+1}(B')$ acts transitively on $Um_{r+1}(B')$. Then given $(b,p) \in Um(B \oplus P, s^2B),$ there exists  $\varepsilon \in E(B \oplus P)$ such that $\varepsilon(b,p)=(1,0)$.
\end{lemma}

A commutative diagram of rings
\begin{figure}[h]
	\[
	\begin{tikzcd}[row sep=2em, column sep=2.5em]
	R \arrow[rr,"g_1"] \arrow[dd,swap,"g_2"] &&
	R_1 \arrow[dd,"f_1"] \\
	&  \\
	R_2 \arrow[rr,swap, "f_2"] && R'
	\end{tikzcd}
	\]\label{fig1}
\end{figure}

has \textit{Milnor patching property for unimodular rows} if for every $r \geq 3$ and $v_1 \in Um_r(R_1)$ such that $f_1(v_1)\simm{R'}e_1,$ there exists a pullback $v \in R$ such that $g_1(v)\simm{R_1}v_1$. 

A \textit{Generalized Karoubi square} is a commutative square where $R_2 = S^{-1}R,$ $R' = g_1(S)^{-1}R_1$ and the map $g_1$ is an analytic epimorphism along $S \subset R$.

Gubeladze (\cite{Gubeladze3-MR1206629}$,$ Proposition 9.1) proved that a generalized Karoubi square has Milnor patching property for unimodular rows. This plays a vital role in patching of unimodular elements. 

A monoid $M$ is a set with an associative operation $M \times M \rightarrow M$ with unity. We will use  multiplicative notation for the operations in $M$. Monoids isomorphic to $\mathbb{Z}_+^r$ are called \textit{free} monoids of rank $r$. One can refer to (\cite{Gubeladze1-MR2508056}$,$ Chapter 2) for a detailed read. Given a ring $R,$ similar to that of a polynomial algebra$,$ we talk about the \textit{monoid algebra} $R[M],$ generated as a free $R$-module with basis as elements of $M$ and coefficients in $R$. By $R[d,m,n][M],$ we mean a monoid algebra with coefficients in a ring $A$ of type $R[d,m,n],$ i.e. $A[M]$.


\section{Cancellation of Projective Modules}

\begin{proposition}\label{p1}
	Let $A$ be a ring of type $R[d,m,n]$ and $M$ be a commutative cancellative monoid. Then $E_r(A[M])$ acts transitively on $Um_r(A[M]),$ if $r \geq \max \{3, d+2\}$.
\end{proposition}

\begin{proof}
	We use induction on $n$. For $n=0,$ the result follows from Theorem $\ref{t1}$ by choosing the monoid as $M[X_1, \ldots,X_m] \simeq M \oplus \mathbb{Z}_+^m$.
	
	Assume $n>0$. For the ring $A$ of the first type where each $l_i$ is a variable $X_{i_j},$ we may assume that $l_n = X_m.$ Consider the multiplicative subset $S = 1+ f_nR[X_m]$ and write $A_S = R'[d,m-1,n-1] = R'[X_1, \ldots, X_{m-1}, {f_1(l_1)}^{-1}, \ldots, {f_{n-1}(l_{n-1})}^{-1}],$ where $R' =  R[X_m]_{f_nS}$. Let $a = (a_1, \ldots, a_r) \in Um_r(A[M])$. As $dim R' = d,$ by induction on $n,$  $a\simm{ A_S[M]}e_1 = (1, 0, \ldots, 0).$ Choose an $s \in S$ such that there exists  $\sigma' \in E_r(A_s[M])$ with $\sigma'(a) = e_1$. 
	
	Consider the following fibre product diagram
	\begin{equation*}
	\begin{tikzcd}
	C[M] \ar{r}{p} \ar{d}{} & A[M] \ar{d}{} \\ 
	C_s[M] \ar{r}{} & A_s[M]
	\end{tikzcd}
	\end{equation*}
	where $C =  R[X_1, \ldots, X_{m}, {f_1(l_1)}^{-1}, \ldots, {f_{n-1}(l_{n-1})}^{-1}]$ is of type  $R[d,m,n-1]$ and $A=C_{f_n}$. As the diagram above has Milnor patching property for unimodular rows$,$ there exists a $c \in Um_r(C[M])$ such that $p(c) \simm{A[M]}a$. By induction on $n,$ $c\simm{C[M]}e_1$ and hence their respective images $a\simm{A[M]}e_1$.
	
	The proof for the ring of the second type follows through if the following reduction is considered. Let $l_n = \sum k_iX_i - r,$ where $k_i \in k$ are not all zero and $r \in R$. Choose a $\phi \in E_m(k),$ such that $(k_1, \ldots, k_m) = (0, \ldots, 0, 1)\phi $. By changing the variables $(X_1,\ldots, X_m)$ to $\phi(X_1,\ldots,X_m),$ we can assume $l_n = X_m + r$. Again transforming the variable $X_m$ using the translation $X_m\mapsto X_m - r,$ we may assume that $l_n = X_m$. Use the  above arguments to complete the proof.
\end{proof}	

Following is a direct consequence of the above result$,$ proof of which follows in spirit to that of (\cite{AMD&MKK-MR2564858}$,$ Theorem 3.8)

\begin{corollary}\label{c2}
	Let $A$ be a ring of type $R[d,m,n]$ and $M$ be a commutative cancellative monoid. Then the canonical map $\phi_r: GL_r(A[M])/E_r(A[M]) \rightarrow K_1(A[M])$ is surjective for $r \geq \{2, dimR +1\}.$
\end{corollary}

The following corollary follows directly from Lemma \ref{l2} by utilizing Proposition \ref{p1}.

\begin{corollary}\label{c1}
	Let $A$ be a ring of type $R[d,m,n]$ and $M$ be a commutative cancellative torsion-free seminormal monoid. Let $P$ be a projective $A[M]$-module of rank $ r$. Then there exists $s\in R$
	satisfying the properties of Lemma \ref{l1}.
\end{corollary}

\begin{proof}
	It is enough to show that there exists an $s\in R$ such that $P_s$ is free. We can assume that $R$ is reduced.
	Let $S$ be the set of non-zerodivisors of $R$. Then $S^{-1}R$ is a direct product of fields. Hence without loss of generality$,$ we may assume that $S^{-1}R$ is a field. By Gabber (\cite{Gabber-MR1891173}$,$ Theorem 2.1)$,$ all projective $S^{-1}A$ modules are free. By Swan (\cite{Rao-MR955292}$,$ Corollary 1.3)$,$ projective modules over $S^{-1}A[M]$ are extended from $S^{-1}A,$ hence are free. Thus we can choose an $s\in R$ such that $P_s$ is free. 
\end{proof}

\begin{theorem}\label{t4}
	Let $A$ be a ring of type $R[d,m,n]$ and $M$ be a commutative cancellative torsion-free seminormal monoid. Let $P$ be a projective $A[M]$-module of rank $r \geq \max \{2, d+1\}$. Then $E(A[M] \oplus P)$ acts transitively on $Um(A[M] \oplus P)$. In particular$,$ P is cancellative.
\end{theorem}

\begin{proof} 
	Without loss of generality$,$ we can assume $A$ is reduced with connected spectrum. If $d=0,$ then  $P$ is free by Corollary \ref{c1} and the result follows from Proposition \ref{p1}. 
	
	Assume $d>0$. Let $(a,p) \in Um(A[M] \oplus P)$. By Corollary \ref{c1}$,$ choose a non-zerodivisor $s\in R$ satisfying the hypothesis of the Lemma \ref{l1} and let $``-"$ denote reduction modulo $s^2A[M]$. Then by induction on $d,$ $(\bar{a}, \bar{p})\simm{A[M]/s^2A[M]}(1,0)$. Since an element of $E(\overline {A[M]} \oplus \bar {P})$ can be lifted to an element of $E(A[M]\oplus P),$ we may assume that $(a,p) \in Um(A[M] \oplus P, s^2A[M])$. By Lemma \ref{l2}$,$ $(a, p)\simm{A[M] \oplus P}(1,0)$ and hence $E(A[M] \oplus P)$ acts transitively on $Um(A[M] \oplus P)$. 
\end{proof}


\section{Generalization of Swan's result}

Using techniques similar to that of the previous section$,$ the following can be derived.

\begin{proposition}\label{p2}
	Let $R$ be a UFD of dimension $1,$ $M$ be a commutative cancellative torsion-free seminormal monoid and $A = R[1,m,n]$. Then all projective $A[M]$-modules are free.
\end{proposition}

\begin{proof}
	Let $P$ be a projective $A[M]$-module. We will induct on $n$. If $n = 0,$ then $A[M] = A[M \oplus \mathbb{Z}_+^m]$. By \cite{Gubeladze2-MR937805}$,$ $P$ is free.
	
	Assume $n>0$. First assume that $A$ is of the type where all $l_i$ are variables. We can assume $l_n = X_m$. Consider the multiplicative subset $S = 1+ f_nR[X_m]$ and rewrite $A_S = R'[1, m-1, n-1] = R'[X_1, \ldots, X_{m-1}, {f_1(l_1)}^{-1}, \ldots, {f_{n-1}(l_{n-1})}^{-1}],$ where $R' =  R[X_m]_{f_nS}$ is a $1$-dimensional UFD. By induction on $n,$ $P_S$ is free$,$ choose a $g \in R[X_m]$ such that $P_{1+f_ng}$ is free.
	
	Let $C=R[X_1, \ldots, X_{m}, {f_1(l_1)}^{-1}, \ldots, {f_{n-1}(l_{n-1})}^{-1}]$ be the subring of $A$ of type $R[1,m,n-1]$. Then $A=C_{f_n}$.  By Milnor patching$,$ we get $P$ is extended from  $C[M]$. By induction on $n,$ projective $C[M]$-modules are free. Therefore$,$ $P$ is free. 
	
	The proof when $l_i$ are of the second type follows in a similar fashion to that of Proposition \ref{p1}.   
\end{proof}	

Swan's criterion for a non local ring $R$ can be condensed and simply put as: A commutative domain $R$ is an element of the collection $\mathcal{R}_n$ if all projective modules of rank $n$ over $R_\mathfrak{m} [x,x^{-1}]$ and $R(t)_\mathfrak{n}[x, x^{-1}]$ are free$,$ where $\mathfrak{m} \in max(R)$ and  $\mathfrak{n} \in max(R(t))$. The following theorem due to Popescu (\cite{Popescu-MR986438}$,$ Theorem 3.1) helps us visualize this collection better using Theorem \ref{t6}. A ring $R$ is \textit{essentially of finite type} over a ring $S,$ if $R$ is the localization of an affine $S$-algebra $T$ at a multiplicatively closed subset of $T$.

\begin{theorem}\label{t5}
	Let $R$ be a regular local ring containing a field $k $. Then $R$ is a filtered inductive limit of regular local rings essentially of finite type over $\mathbb{Z}$.
\end{theorem}

\begin{theorem}\label{t8}
	The class of regular domains containing a field belongs to $\mathcal{R}_n$ for all $n>0$.  
\end{theorem}

\begin{proof}
	It is enough to show that if $R$ is a regular local ring containing a field $k,$ then projective $R[X,X^{-1}]$-modules are free. If $P$ is a projective $R[X,X^{-1}]$-module$,$ then by Theorem \ref{t5}$,$ we may assume that $R$ is a regular $\mathbb{Z}$-spot and in particular a regular spot over the prime subfield of $k$. By Swan's result \cite{Swan-MR1144038}$,$ $P$ is free.
\end{proof}	

The following theorem can be found in (\cite{Swan-MR1144038}$,$ Theorem 1.2)

\begin{theorem}\label{t-3}
	Let $R \in \mathcal{R}_n$ and $M$ be a commutative cancellative torsion-free seminormal monoid with $U$ as its group of units. Then the following is a patching diagram
	\begin{figure}[h]
		\[
		\begin{tikzcd}[row sep=2em, column sep=2.5em]
		\mathcal{P}_n(R) \arrow[rr,""] \arrow[dd,swap,""] &&
		\mathcal{P}_n(R[M]) \arrow[dd,""] \\
		&  \\
		\mathcal{P}_n(R) \arrow[rr,swap, ""] && \mathcal{P}_n(RU)
		\end{tikzcd}
		\]\label{fig2}
	\end{figure}
\end{theorem}

This theorem gives us a way to see when projective modules over $R[M]$ are extended from $R$. If $U$ is trivial$,$ then all projective $R[M]$-modules of rank $n$ are extended from $R$. Also$,$ $U$ being torsion free$,$ is a filtered limit of finite rank free abelian groups. This leads to applications$,$ when sufficient information regarding $\mathcal{P}_n(R[\mathbb{Z}^r])$ is provided. 

Most results of \cite{MKK&SAL-MR3153610} proved for the regular ring $B$ can be generalized to $B[M],$ where  $M$ is a commutative cancellative torsion-free seminormal monoid. We state some results 
generalizing the results in Section $4$ and $5$ of \cite{MKK&SAL-MR3153610}. 

\begin{theorem}\label{t6}
	Let $A$ be a ring of type $R[d,m,n]^*,$ where $(R, \mathfrak{m},K)$ is a regular local ring containing a field $k$ such that either char $k = 0$ or char $k=p$ and tr-deg $K/\mathbb{F}_p \geq 1$. Let $M$ be a commutative cancellative torsion-free seminormal monoid. Then
	\begin{enumerate} 
		\item all projective $A[M]$ and $A[M] \otimes_R R(T)$-modules are free;
		\item if $f, g\in R$ form a part of a regular system of parameters$,$ then all projective $A[M]_f$ and $A[M]_{fg}$-modules are free;
		\item if $g_1, \ldots, g_t \in R$ form a part of a regular system of parameters$,$ then projective $A[M]_{{g_1}\ldots {g_t}}$-modules of rank $\geq t$ are free. 
	\end{enumerate} 
\end{theorem}

\begin{proof}
	Let $B$ be any one of the rings $A, A\otimes R(T), A_f, A_{fg}, A_{g_1,\ldots ,g_t}$. Then $B$ is a regular ring containing a field$,$ hence belongs to $\mathcal R_n$ for all $n>0$ by Theorem \ref{t8}.
	When $M=0,$ we are done by \cite{MKK&SAL-MR3153610}. Let $U$ be the group of units of $M$. If we show that projective $B[U]$-modules are extended from $R,$ then by Theorem \ref{t-3}$,$ projective $B[M]$-modules are extended from $B$. Hence the conclusion will again follow from \cite{MKK&SAL-MR3153610}.
	
	Since $M$ and hence $U$ is torsion-free$,$ $U$ is direct limit of finite rank free abelian groups. Then $B[\mathbb Z^r]$ is of the type $R[d,m+r,n+r]^*,$ the result follows from \cite{MKK&SAL-MR3153610}.
\end{proof}


\bibliographystyle{abbrv}
\bibliography{paper.bib}

\end{document}